\RequirePackage{amsmath}
\documentclass[runningheads]{llncs}
\usepackage[hyphens]{url}
\usepackage[hyphenbreaks]{breakurl}
\usepackage[colorlinks, urlcolor=RoyalBlue, linkcolor=Red, citecolor=Green]{hyperref} 

\usepackage[
	backend=biber,
	natbib=true, 
	style=numeric, 
	sorting=nyt, 
	maxbibnames=10, 
	giveninits=true,
	backref=true, 
	url=false,
	doi=false,
	isbn=false,
	arxiv=abs, 
	date=year, 
	abbreviate = true 
]{biblatex}
\AtEveryBibitem{
	\clearfield{language}
	\iffieldundef{eprinttype}{}{
		\clearfield{url} 
		\clearfield{urldate} 
		\clearfield{urlyear}
		\clearfield{urlmonth}
	}
}

\usepackage{graphicx} 
\usepackage[dvipsnames]{xcolor} 
\usepackage{adjustbox}

\usepackage{tikz} 
\usetikzlibrary{
	cd, 
	petri, 
	backgrounds, 
	arrows, 
	positioning, 
	decorations.markings, 
	calc,  
	fit, 
	shapes.multipart
}

\usepackage{mathtools} 
\usepackage{amssymb} 

\usepackage{amsthm} 
\usepackage{stmaryrd} 

\usepackage{relsize} 
\usepackage{microtype} 
\usepackage{multicol} 
\usepackage{csquotes} 
\usepackage{xspace}
\usepackage{enumitem}
\usepackage{fontawesome} 

\usepackage{tabularx, cellspace} 
\usepackage[capitalize]{cleveref}
\newcounter{theoremUnified} 
\numberwithin{theoremUnified}{section} 
\numberwithin{theoremUnified}{section} 

\newtheoremstyle{plainStyle} 
{2mm} 
{2mm} 
{} 
{} 
{\bfseries} 
{.} 
{.5em} 
{} 

\newtheoremstyle{italicStyle} 
{2mm} 
{2mm} 
{\itshape} 
{} 
{\bfseries} 
{.} 
{.5em} 
{} 

\newtheorem{notation}{Notation}{\itshape}{\rmfamily}
{\itshape}{\rmfamily}
\theoremstyle{plainStyle}

\newcommand{\Naturals}{\mathbb{N}} 

\newcommand{\Mset}[1]{{#1}^{\Naturals}} 
\newcommand{\Msets}[1]{{#1}^{\oplus}} 

\usepackage{tensor}

\newcommand{\Id}[1]{\text{id}_{#1}} 

\usepackage{eucal}

\newcommand{\CategoryC}{\mathcal{C}}
\newcommand{\CategoryD}{\mathcal{D}}

\newcommand{\Fun}[1]{{{#1}^\sharp}} 
\newcommand{\NetSem}[1]{\left( #1, \Fun{#1}\right)} 

\newcommand{\Semantics}{\mathcal{S}} 

\newcommand{\Comm}[1]{\mathfrak{C}\kern-.1em\left(#1\right)} 
\newcommand{\CommB}[1]{\mathfrak{C}_B\kern-.1em\left(#1\right)} 
\newcommand{\CommM}[1]{\mathfrak{C}_M\kern-.2em\left(#1\right)} 

\newcommand{\Grothendieck}[1]{\textstyle\int{#1}} 
\newcommand{\GrothendieckS}[1]{\Grothendieck{\Fun{#1}}} 

\newcommand{\harpvecsign}{\scriptscriptstyle\rightharpoonup}
\newcommand{\harpoonvec}[2]{%
	\ifx\displaystyle#1\doalign{$\harpvecsign$}{#1#2}\fi
	\ifx\textstyle#1\doalign{$\harpvecsign$}{#1#2}\fi
	\ifx\scriptstyle#1\doalign{\scalebox{.6}[.9]{$\harpvecsign$}}{#1#2}\fi
	\ifx\scriptscriptstyle#1\doalign{\scalebox{.5}[.8]{$\harpvecsign$}}{#1#2}\fi
}
\newcommand{\doalign}[2]{%
	{\vbox{\offinterlineskip\ialign{\hfil##\hfil\cr#1\cr$#2$\cr}}}%
}

\newcommand{\Set}{\mathbf{Set}} 
\newcommand{\Span}{\mathbf{Span}} 
\newcommand{\Cat}{\mathbf{Cat}} 

\newcommand{\Petri}{\mathbf{Petri}} 

\newcommand{\CSMC}{\mathbf{CSMC}} 
\newcommand{\FCSMC}{\mathbf{FCSMC}} 

\newcommand{\PetriS}[1]{\mathbf{Petri}^{#1}} 
\newcommand{\PetriSpan}{\PetriS{\Span}} 
\newcommand{\PetriMana}{\PetriS{\mathcal M}} 

\newcommand{\TensorUnit}{I} 

\newcommand{\Suchthat}[2]{\left\{#1 \: \middle\vert \: #2\right\}} 

\def\backgrnd{black!10}	

\pgfdeclarelayer{edgelayer}
\pgfdeclarelayer{nodelayer}
\pgfdeclarelayer{background}
\pgfsetlayers{background, edgelayer,nodelayer,main}

\tikzstyle{place}=
[circle,thick,draw=blue!75,fill=blue!20,minimum size=6mm]
\tikzstyle{manaplace}=
[circle,thick,draw=green!75,fill=green!20,minimum size=6mm]
\tikzstyle{antiplace}=
[circle,thick,draw=red!75,fill=red!20,minimum size=6mm]
\tikzstyle{transition}=
[rectangle,thick,draw=black!75,fill=black!20,minimum size=4mm]
\tikzstyle{inarrow}=[->, >=stealth, shorten >=.03cm,line width=1.5]
\tikzstyle{antiarrow}=[<-, red!75,  >=stealth, shorten >=.03cm,line width=1.5]

\tikzset{
	pics/netA/.style args={#1/#2/#3/#4/#5/#6/#7}{code={

					\node [place,label=above:$p_1$, tokens={%
								#1
							}] (-pl_1) {};

					\node [transition,label=above:$t$, label=below:#5] (-tr_1) [right = of -pl_1] {};

					\node [place,label=above:$p_2$, tokens={%
								#2
							}] (-pl_2) [right = of -tr_1] {};

					\node [transition,label=left:$v$, label=above:#6] (-tr_2) [below = of -tr_1] {};
					\node [transition,label=below:$u$, label=above:#7] (-tr_3) [below = of -tr_2] {};

					\node [place,label=below:$p_3$, tokens={%
								#3
							}] (-pl_3) [left = of -tr_3] {};

					\node [place,label=below:$p_4$, tokens={%
								#4
							}] (-pl_4) [right = of -tr_3] {};

					\draw[inarrow] (-pl_1) -- (-tr_1);
					\draw[inarrow] (-tr_1) -- (-pl_2);
					\draw[inarrow] (-pl_2) -- (-tr_2);
					\draw[inarrow] (-tr_2) -- (-pl_3);
					\draw[inarrow] (-tr_2) -- (-pl_4);
					\draw[inarrow] (-pl_3) -- (-tr_3);
					\draw[inarrow] (-tr_3) -- (-pl_4);
				}}
}

\tikzset{ 
	oriented WD/.style={
			every to/.style={
					out=0,in=180,draw
				},
			label/.style={
					font=\everymath\expandafter{\the\everymath\scriptstyle},
					inner sep=0pt,
					node distance=2pt and -2pt
				},
			semithick,
			node distance=1 and 1,
			decoration={
					markings, mark=at position \stringdecpos with \stringdec
				},
			ar/.style={
					postaction={decorate}
				},
			execute at begin picture={
					\tikzset{
						x=\bbx, y=\bby,
						every fit/.style={
								inner xsep=\bbx, inner ysep=\bby
							}
					}
				}
		},
	string decoration/.store in=\stringdec,
	string decoration={
			\arrow{stealth};
		},
	string decoration pos/.store in=\stringdecpos,
	string decoration pos=.7,
	bbx/.store in=\bbx,
	bbx = 1.5cm,
	bby/.store in=\bby,
	bby = 1.5ex,
	bb port sep/.store in=\bbportsep,
	bb port sep=1.5,
	bb port length/.store in=\bbportlen,
	bb port length=4pt,
	bb penetrate/.store in=\bbpenetrate,
	bb penetrate=0,
	bb min width/.store in=\bbminwidth,
	bb min width=1cm,
	bb rounded corners/.store in=\bbcorners,
	bb rounded corners=2pt,
	bb small/.style={
			bb port sep=1,
			bb port length=2.5pt,
			bbx=.4cm, bb min width=.4cm,
			bby=.7ex
		},
	bb medium/.style={
			bb port sep=1,
			bb port length=2.5pt,
			bbx=.4cm,
			bb min width=.4cm,
			bby=.9ex
		},
	bb/.code 2 args={
			\pgfmathsetlengthmacro{\bbheight}{\bbportsep * (max(#1,#2)+1) * \bby}
			\pgfkeysalso{
				draw,
				minimum height=\bbheight,
				minimum width=\bbminwidth,
				outer sep=0pt,
				rounded corners=\bbcorners,
				thick,
				prefix after command={
						\pgfextra{\let\fixname\tikzlastnode}
					},
				append after command={
						\pgfextra{
							\draw
							\ifnum #1=0
								{}
							\else
								foreach \i in {1,...,#1} {
								($(\fixname.north west)!{\i/(#1+1)}!(\fixname.south west)$) +(-
								\bbportlen,0)
								coordinate (\fixname_in\i) -- +(\bbpenetrate,0) coordinate (\fixname_in\i')
								}
							\fi
							\ifnum
								#2=0{}
							\else
								foreach \i in {1,...,#2} {
								($(\fixname.north east)!{\i/(#2+1)}!(\fixname.south east)$) +(-
								\bbpenetrate,0)
								coordinate (\fixname_out\i') -- +(\bbportlen,0) coordinate (\fixname_out\i)
								}
							\fi;
						}
					}
			}
		},
	bb name/.style={
			append after command={
					\pgfextra{
						\node[anchor=north] at (\fixname.north) {#1}
						;}
				}
		}
}

\def\lrc{
	\begin{tikzpicture}[scale=.33]
		\draw[-] (0,0) -| (1,1);
	\end{tikzpicture}
}

\newcommand{\pb}{\arrow[dr, phantom, "\lrc", very near start]}
\tikzset{commutative diagrams/.cd,arrow style=tikz,diagrams={>=stealth'}}

\usepackage{lineno}
\begin{document}
\title{Nets with Mana: A Framework for Chemical Reaction Modelling}
\author{
	Fabrizio Genovese\inst{1}\,\textsuperscript{\faEnvelopeO}\orcidID{0000-0001-7792-1375}\\
	Fosco Loregian\inst{2}\,\textsuperscript{\faEnvelopeO}\orcidID{0000-0003-3052-465X}\\
	Daniele Palombi\inst{3}\,\textsuperscript{\faEnvelopeO}\orcidID{0000-0002-8107-5439}
}
\authorrunning{F. Genovese \and F. Loregian \and D. Palombi}
\institute{%
	University of Pisa\\
	\email{fabrizio.romano.genovese@gmail.com}
	\and
	Tallinn University of Technology\\
	\email{fosco.loregian@gmail.com}
	\and
	Sapienza University of Rome\\
	\email{danielepalombi@protonmail.com}
}
\maketitle
\begin{abstract}
	We use categorical methods to define a new flavor of Petri nets
	where transitions can only fire a limited number of times, specified by
	a quantity that we call mana. We do so with chemistry in mind, looking at
	ways of modelling the behavior of chemical reactions
	that depend on enzymes to work. We prove that such nets can be either obtained as a result
	of a comonadic construction, or by enriching them with extra information
	encoded into a functor. We then use a well-established categorical result to
	prove that the two constructions are equivalent, and generalize them to
	the case where the firing of some transitions can "regenerate" the mana of others.
	This allows us to represent the action of catalysts and also of biochemical processes where
	the byproducts of some chemical reaction are exactly the enzymes that another reaction
	needs to work.
\end{abstract}
\paragraph*{\bf Acknowledgements}
The first author was supported by the project MIUR PRIN 2017FTXR7S “IT-MaTTerS” and by the \href{https://gitcoin.co/grants/1086/independent-ethvestigator-program}{Independent Ethvestigator Program}.

The second author was supported by the ESF funded Estonian IT Academy research measure (project 2014-2020.4.05.19-0001).

\bigskip\noindent
A video presentation of this paper can be found on Youtube at \href{https://www.youtube.com/watch?v=9sxVBJs1okE}{9sxVBJs1okE}.

\section{Introduction}\label{sec: introduction}
Albeit they have found great use outside their original domain,
Petri nets were invented to describe chemical reactions~\cite{Petri2008}.
The interpretation is as simple as it can get: places of the net represent
types of compounds (be it atoms or molecules); tokens represent the amount of
each combination we have available; transitions represent reactions transforming
compounds.
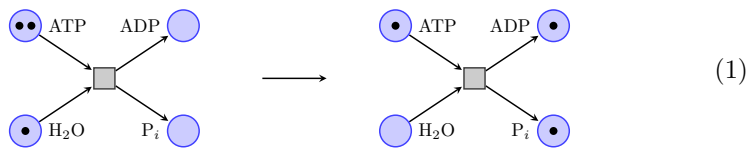
\begin{linenomath*}
	\begin{equation}
		\scalebox{0.7}{
	\begin{tikzpicture}[baseline=(current bounding box.center)]
		\begin{scope}[xshift=0]
			\begin{pgfonlayer}{nodelayer}
				\node [place,tokens=2, label=right:ATP] (1a) at (-1.5,1) {};
				\node [place,tokens=1, label=right:$\text{H}_2\text{O}$] (1b) at (-1.5,-1) {};
				\node [place,tokens=0, label=left:ADP] (3a) at (1.5,1) {};
				\node [place,tokens=0, label=left:$\text{P}_i$] (3b) at (1.5,-1) {};
				\node[transition] (2a) at (0,0) {};
			\end{pgfonlayer}
			\begin{pgfonlayer}{edgelayer}
				\draw[style=inarrow, thick] (1a) to (2a);
				\draw[style=inarrow, thick] (1b) to (2a);
				\draw[style=inarrow, thick] (2a) to (3a);
				\draw[style=inarrow, thick] (2a) to (3b);
			\end{pgfonlayer}
		\end{scope}

		\begin{scope}[xshift=200]
			\begin{pgfonlayer}{nodelayer}
				\node [place,tokens=1, label=right:ATP] (1a) at (-1.5,1) {};
				\node [place,tokens=0, label=right:$\text{H}_2\text{O}$] (1b) at (-1.5,-1) {};
				\node [place,tokens=1, label=left:ADP] (3a) at (1.5,1) {};
				\node [place,tokens=1, label=left:$\text{P}_i$] (3b) at (1.5,-1) {};
				\node[transition] (2a) at (0,0) {};
			\end{pgfonlayer}
			\begin{pgfonlayer}{edgelayer}
				\draw[style=inarrow, thick] (1a) to (2a);
				\draw[style=inarrow, thick] (1b) to (2a);
				\draw[style=inarrow, thick] (2a) to (3a);
				\draw[style=inarrow, thick] (2a) to (3b);
			\end{pgfonlayer}
		\end{scope}



		\draw[style=inarrow, thick] (3,0) -- (4.25,0);

	\end{tikzpicture}
}
	\end{equation}
\end{linenomath*}
Still, things are not so easy in real-world chemistry: reactions often need ``context'' to happen,
be it a given temperature, energy, presence of enzymes and catalysts. This is particularly true
in biochemical processes, where enzymes of all
sorts mediate rather complicated reactions. Importantly, these enzymes tend to degrade over time, resulting in reactions that do not keep happening forever~\cite{Latelier2006}. This is one of the (many) reasons why organisms wither and die,
but it is not captured by the picture above, where the transition can fire every time it is enabled.

Borrowing the terminology from the popular Turing machine \emph{Magic: The gathering}~\cite{Wikipedia2020,Churchill2019}
we propose a possible solution to this problem by endowing transitions in a net with \emph{mana}~\cite{Wikipedia2020a},
representing the ``viability'' of reactions: once a reaction is out of mana, it cannot fire anymore.
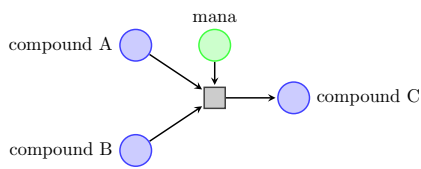
\begin{linenomath*}
	\begin{equation*}
		\scalebox{0.7}{
	\begin{tikzpicture}
		\begin{scope}[xshift=0]
			\begin{pgfonlayer}{nodelayer}
				\node [place,tokens=0, label=left:compound A] (1a) at (-1.5,1) {};
				\node [place,tokens=0, label=left:compound B] (1b) at (-1.5,-1) {};
				\node [place,tokens=0, label=right:compound C] (3a) at (1.5,0) {};

				\node [manaplace,tokens=0, label=above:mana] (2a) at (0,1) {};
				\node[transition] (2b) at (0,0) {};
			\end{pgfonlayer}
			\begin{pgfonlayer}{edgelayer}
				\draw[style=inarrow, thick] (1a) to (2b);
				\draw[style=inarrow, thick] (1b) to (2b);
				\draw[style=inarrow, thick] (2b) to (3a);
				\draw[style=inarrow, thick] (2a) to (2b);
			\end{pgfonlayer}
		\end{scope}

	\end{tikzpicture}
}
	\end{equation*}
\end{linenomath*}
Now, we could just represent mana by adding another place for each transition in a net.
Indeed, this is the idea we will start with. Still, being accustomed to the \emph{yoga} of type-theoretic reasoning,
we are also aware that throwing everything in the same bucket is rarely a good idea: albeit mana can be a chemical compound, it is more realistic to consider it as conceptually separated from the reactions it catalyzes.

\medskip
Resorting to categorical methods, we show how we can axiomatize the idea of mana in a better
way. We do so by relaxing the definitions in the categorical approach to coloured nets already developed in~\cite{Genovese2020},
defining a functorial semantics representing the equipment of a net with mana.
Then, we will prove how categorical techniques allow us to internalize such
a semantics, exactly obtaining what we represented in the picture above.

Finally, we will show how the categorical semantics naturally leads to a further generalization,
where transitions not only need mana to function but also provide byproducts that can be used as
mana for other transitions. This allows us to represent \emph{catalysts}\footnote{
	An unrelated categorical approach to nets with catalysts can be found in~\cite{Baez2019a}.
} (i.e. cards `adding $\infty$ to the mana pool', or more precisely mana that
does not deteriorate over time) and in general nets apt to describe \emph{two-layered} chemical processes,
the first layer being the usual one represented by Petri nets and the second layer being the one of
enzymes and catalysts being consumed and exchanged by different reactions.

\section{Nets and their executions}\label{sec: nets and their executions}
Before presenting the construction itself, it is worth recapping the main points about
categorical semantics for Petri nets. The definition of net commonly used in the
categorical line of work is the following:
\begin{notation}
	Let $S$ be a set; denote with $\Msets{S}$ the set of \emph{finite} multisets over $S$.
	Multiset sum will be denoted with $\oplus$, multiplication with $\odot$ and
	difference (only partially defined) with
	$\ominus$. $\Msets{S}$ with $\oplus$ and the empty multiset is isomorphic to the free commutative monoid on $S$.
\end{notation}
\begin{definition}[Petri net]\label{def: Petri net}
	We define a
	\emph{Petri net} as a couple functions $T \xrightarrow{s,t} \Msets{S}$ for
	some sets $T$ and $S$, called the set of places and transitions of the net, respectively.

	A \emph{morphism of nets} is a couple of functions $f: T \to T'$ and $g: S \to S'$ such that
	the following square commutes, with $\Msets{g}: \Msets{S} \to \Msets{S'}$ the obvious
	lifting of $g$ to multisets:
	\begin{linenomath*}
		\begin{equation*}
			\begin{tikzcd}
				{\Msets{S}} & {T} & {\Msets{S}} \\
				{\Msets{S'}} & {T'} & {\Msets{S'}}
				\arrow["{s}"', from=1-2, to=1-1]
				\arrow["{s'}", from=2-2, to=2-1]
				\arrow["{t'}"', from=2-2, to=2-3]
				\arrow["{t}", from=1-2, to=1-3]
				\arrow["{\Msets{g}}"', from=1-1, to=2-1]
				\arrow["{\Msets{g}}", from=1-3, to=2-3]
				\arrow["{f}" description, from=1-2, to=2-2]
			\end{tikzcd}
		\end{equation*}
	\end{linenomath*}
	Petri nets and their morphisms form a category, denoted $\Petri$. The reader can find additional details in~\cite{Meseguer1990}.
\end{definition}
\begin{definition}[Markings and firings]\label{def: Petri net firing}
	A \emph{marking} for a net  $T \xrightarrow{s,t} \Msets{S}$ is an element of $\Msets{S}$,
	representing a distribution of tokens in the net places. A transition $u$ is \emph{enabled} in a marking $M$ if
	$M \ominus s(u)$ is defined. An enabled transition can \emph{fire}, moving tokens in the net.
	Firing is considered an atomic event, and the marking resulting from firing $u$ in $M$ is $M \ominus s(u) \oplus t(u)$.
\end{definition}
Category theory provides a slick definition to represent all the possible
executions of a net -- all the ways one can fire transitions starting from a given marking -- as morphisms in a category. There are various ways to do this~\cite{Meseguer1990,Master2020,Sassone1995,Genovese2019c,Genovese2019b, Baez2021},
depending if we want to consider
tokens as indistinguishable (common-token philosophy) or not (individual-token philosophy).
In this work, we focus on chemical reactions. Since we consider atoms and molecules of the same kind to be physically indistinguishable, we will adopt the common-token perspective. In this case, the category of executions of a net is a \emph{commutative monoidal category} -- a
monoidal category whose monoid of objects is commutative.
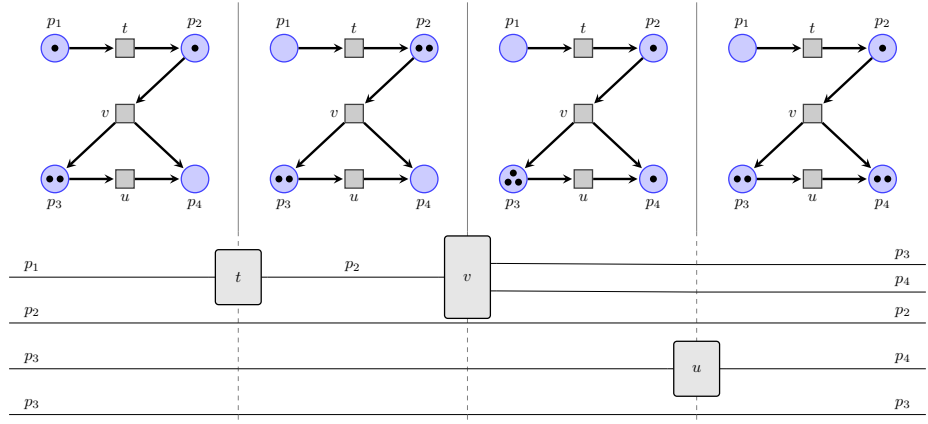
\begin{center}
	\begin{adjustbox}{max width=\textwidth}
		\begin{tikzpicture}
			\pgfmathsetmacro\bS{5}
			\pgfmathsetmacro\hkX{(\bS/3.5)}
			\pgfmathsetmacro\kY{-1.5}
			\pgfmathsetmacro\hkY{\kY*0.5}
			\draw pic (m0) at (0,0) {netA={{1}/{1}/{2}/{0}/{}/{}/{}}};
			\draw pic (m1) at (\bS,0) {netA={{0}/{2}/{2}/{0}/{}/{}/{}}};
			\draw pic (m2) at ({2 * \bS},0) {netA={{0}/{1}/{3}/{1}/{}/{}/{}}};
			\draw pic (m3) at ({3 * \bS},0) {netA={{0}/{1}/{2}/{2}/{}/{}/{}}};
			\begin{scope}[very thin]
				\foreach \j in {1,...,3} {
						\pgfmathsetmacro \k { \j * \bS - 1 };
						\draw[gray,dashed] (\k,-4) -- (\k,-8.25);
						\draw[gray] (\k,1) -- (\k,-4);
					}
			\end{scope}
			\begin{scope}[shift={(0,-4)}, oriented WD, bbx = 1cm, bby =.4cm, bb min width=1cm, bb port sep=1.5]
				\draw node [fill=\backgrnd,bb={1}{1}] (Tau) at (\bS -1,-1) {$t$};
				\draw node [fill=\backgrnd,bb={1}{2}, ] (Mu)  at ({2 * \bS - 1},-1) {$v$};
				\draw node [fill=\backgrnd,bb={1}{1}] (Nu)  at ({3 * \bS - 1},{2 * \kY}) {$u$};
				\draw (-1,-1) --     node[above] {$p_1$}       (0,-1)
				--                  node[above] {}          (Tau_in1);
				\draw (-1,-2) -- node[above] {$p_2$} (0,-2) -- (\bS-1, -2);
				\draw (-1,-3) -- node[above] {$p_3$} (0,-3) -- (\bS-1, -3);
				\draw (-1,-4) -- node[above] {$p_3$} (0,-4) -- (\bS-1, -4);
				\draw (Tau_out1) -- node[above] {$p_2$}    (Mu_in1);
				\draw (\bS-1,-2) -- (2*\bS-1, -2);
				\draw (\bS-1,-3) -- (2*\bS-1, -3);
				\draw (\bS-1,-4) -- (2*\bS-1, -4);
				\draw (Mu_out1) --  (3*\bS-1, -0.725);
				\draw (Mu_out2) --  (3*\bS-1, -1.325);
				\draw (2*\bS-1,-2) -- (3*\bS-1, -2);
				\draw (2*\bS-1,-3) -- (Nu_in1);
				\draw (2*\bS-1,-4) -- (3*\bS-1, -4);
				\draw (3*\bS-1,-0.725) to (4*\bS-2, -0.725) -- node[above] {$p_3$} (4*\bS-1, -0.725);
				\draw (3*\bS-1,-1.325) -- (3*\bS,-1.325) to (4*\bS-2, -1.325) -- node[above] {$p_4$} (4*\bS-1, -1.325);
				\draw (3*\bS-1,-2) to (4*\bS-2, -2) -- node[above] {$p_2$} (4*\bS-1, -2);
				\draw (Nu_out1) to (4*\bS-2, -3) -- node[above] {$p_4$} (4*\bS-1, -3);
				\draw (3*\bS-1,-4) to (4*\bS-2, -4) -- node[above] {$p_3$} (4*\bS-1, -4);
			\end{scope}
		\end{tikzpicture}
	\end{adjustbox}
\end{center}
\begin{definition}[Category of executions -- common-token philosophy]\label{def: executions common token philosophy}
	Let $N: T \xrightarrow{s,t} \Msets{S}$ be a Petri net.
	We can generate a \emph{free commutative strict monoidal category (FCSMC)}, $\Comm{N}$, as follows:
	\begin{itemize}
		\item The monoid of objects is $\Msets{S}$. Monoidal product of objects $A,B$, denoted with $A \oplus B$, is given by the multiset sum;
		\item Morphisms are generated by $T$: each $u \in T$ corresponds to a morphism generator
		      $(u,su,tu)$, pictorially represented as an arrow $su \xrightarrow{u} tu$;
		      morphisms are obtained by considering all the formal (monoidal) compositions of generators and identities.
	\end{itemize}
	The readers can find a detailed description of this construction in~\cite{Master2020}.
\end{definition}
As shown in the picture above, objects in $\Comm{N}$ represent markings of a net: $A\oplus A \oplus B$ means ``two tokens in $A$ and one token in $B$''. Morphisms represent executions of a net, mapping markings to markings.
A marking is reachable from another one if and only if there is a morphism between them.

The correspondence between Petri nets and their executions is categorically well-behaved,
defining an adjunction between the category $\Petri$ and the category $\CSMC$ of commutative strict
monoidal categories, with~\cref{def: executions common token philosophy} building the left-adjoint $\Petri \to \CSMC$.
The readers can find additional details in~\cite{Master2020}.

\section{The internal mana construction}\label{sec: the internal mana construction}
The idea presented in the introduction can na\"ively be formalised by just attaching
an extra input place to any transition in a net, representing the mana a given transition
can consume. We call the following construction \emph{internal} because it builds a category
directly, in contrast with an external equivalent construction given in~\cref{def: external mana construction}.
\begin{definition}[Internal mana construction]\label{def: internal mana construction}
	Let $N: T \xrightarrow{s,t} \Msets{S}$ be a Petri net,
	and consider $\Comm{N}$, its corresponding FCSMC.
	The \emph{internal mana construction of $N$}
	is given by the FCSMC $\CommM{N}$ generated as follows:
	\begin{itemize}
		\item The generating objects of $\CommM{N}$ are the coproduct of the generating objects of $\Comm{N}$ and $T$;
		\item For each generating morphism
		      \begin{linenomath*}
			      \begin{equation*}
				      A_1 \oplus \dots \oplus A_n \xrightarrow{u} B_1 \oplus \dots \oplus B_m
			      \end{equation*}
		      \end{linenomath*}
		      in $\Comm{N}$, we introduce a morphism generator in $\CommM{N}$:
		      \begin{linenomath*}
			      \begin{equation*}
				      A_1 \oplus \dots \oplus A_n \oplus u \xrightarrow{u} B_1 \oplus \dots \oplus B_m
			      \end{equation*}
		      \end{linenomath*}
		      Notice that the writing above makes sense because $u$ is an element of $T$.
	\end{itemize}
\end{definition}
Because of the adjunction between $\Petri$ and $\CSMC$,
we can think every FCSMC as being presented by a Petri net.
The category of~\cref{def: internal mana construction}
is presented precisely by the net obtained from $N$ as
we did in~\cref{sec: introduction}: the additional generating objects of
$\CommM{N}$ represent the places containing the mana associated with each transition.
\begin{example}\label{ex: internal mana construction}
	Performing the construction in~\cref{def: internal mana construction}
	on the category of executions of the net on the left gives the category
	of executions of the net on the right, as we expect:
	\begin{linenomath*}
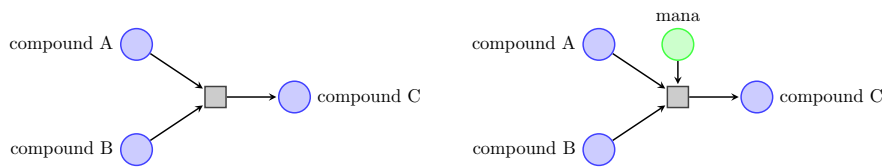
		\begin{equation*}
			\scalebox{0.7}{
	\begin{tikzpicture}
		\begin{scope}[xshift=0]
			\begin{pgfonlayer}{nodelayer}
				\node [place,tokens=0, label=left:compound A] (1a) at (-1.5,1) {};
				\node [place,tokens=0, label=left:compound B] (1b) at (-1.5,-1) {};
				\node [place,tokens=0, label=right:compound C] (3a) at (1.5,0) {};

				\node[transition] (2b) at (0,0) {};
			\end{pgfonlayer}
			\begin{pgfonlayer}{edgelayer}
				\draw[style=inarrow, thick] (1a) to (2b);
				\draw[style=inarrow, thick] (1b) to (2b);
				\draw[style=inarrow, thick] (2b) to (3a);
			\end{pgfonlayer}
		\end{scope}

	\end{tikzpicture}
}
			\quad
			\scalebox{0.7}{
	\begin{tikzpicture}
		\begin{scope}[xshift=0]
			\begin{pgfonlayer}{nodelayer}
				\node [place,tokens=0, label=left:compound A] (1a) at (-1.5,1) {};
				\node [place,tokens=0, label=left:compound B] (1b) at (-1.5,-1) {};
				\node [place,tokens=0, label=right:compound C] (3a) at (1.5,0) {};

				\node [manaplace,tokens=0, label=above:mana] (2a) at (0,1) {};
				\node[transition] (2b) at (0,0) {};
			\end{pgfonlayer}
			\begin{pgfonlayer}{edgelayer}
				\draw[style=inarrow, thick] (1a) to (2b);
				\draw[style=inarrow, thick] (1b) to (2b);
				\draw[style=inarrow, thick] (2b) to (3a);
				\draw[style=inarrow, thick] (2a) to (2b);
			\end{pgfonlayer}
		\end{scope}

	\end{tikzpicture}
}
		\end{equation*}
	\end{linenomath*}
\end{example}
\begin{proposition}\label{prop: comonad}
	The assignment $\Comm{N} \mapsto \CommM{N}$ defines a comonad\footnote{Given a category $\mathcal C$, a \emph{comonad} on $\mathcal C$ is an endofunctor $S$ endowed with two natural transformations $\delta : S \Rightarrow S\circ S$ and $\epsilon : S \Rightarrow 1_{\mathcal C}$ such that $\delta$ is coassociative and has $\epsilon$ as a counit. More succinctly, a comonad is a comonoid in the monoidal category $[\mathcal C,\mathcal C]$ of endofunctors of $\mathcal C$. See \cite[§5.3]{1912.10642} for the definition and a variety of examples.} in
	the category of FCMSCs and strict monoidal functors between them, $\FCSMC$.
\end{proposition}
\begin{proof}
	First of all, we have to prove that the procedure is functorial. For any strict monoidal functor
	$F : \Comm{N} \to \Comm{M}$ we define the action on morphisms $\CommM{F}: \CommM{N} \to \CommM{M}$
	as the following monoidal functor:
	\begin{itemize}
		\item $\CommM{F}$ agrees with $F$ on generating objects coming from $\Comm{N}$. If $u$ is a generating
		      morphism of $\Comm{N}$ and it is $Fu = f$, then $\CommM{F}u = \Mset{f}$, with $\Mset{f}$ being the multiset\footnote{This makes sense since $\Comm{M}$ is free, hence decomposition of morphisms in terms of (monoidal) compositions
			      of generators and identities is unique modulo the axioms of monoidal categories, which do not introduce nor remove generating objects.}
		      counting how many times each generating morphism of $\Comm{M}$ is used in $f$.
		\item $\CommM{F}$ agrees with $F$ on generating morphisms.
	\end{itemize}
	Identities and compositions are clearly respected, making $\CommM{\_}$ an endofunctor
	in $\FCSMC$.
	As a counit, on each component $N$ we define the strict monoidal functor
	$\epsilon_N: \CommM{N} \to \Comm{N}$ sending:
	\begin{itemize}
		\item Generating objects coming from $\Comm{N}$ to themselves, and
		      every other generating object to the monoidal unit.
		\item Generating morphisms to themselves.
	\end{itemize}
	The procedure is natural in the choice of $N$, making $\epsilon$
	into a natural transformation $\CommM{\_} \to \Id{\FCSMC}$.

	As for the comultiplication, on each component $N$ we define the strict monoidal functor
	$\delta_N: \CommM{N} \to \CommM{\CommM{N}}$ sending:
	\begin{itemize}
		\item Generating objects coming from $\Comm{N}$ to themselves,
		      every other generating object $u$ is sent to $u \oplus u$.
		\item Generating morphisms are again sent to themselves.
	\end{itemize}
	The naturality of $\delta$ and the comonadicity conditions are a straightforward check.
\end{proof}

\section{The external mana construction}\label{sec: the external mana construction}
As we stressed in~\cref{sec: introduction}, the construction
as in~\cref{def: internal mana construction} has the disadvantage
of throwing everything in the same bucket: in performing it, we do not keep any more
a clear distinction between the different layers of our chemical reaction networks, given
by mana and compounds.

In the spirit of~\cite{Genovese2020}, we now recast the mana construction \emph{externally},
as Petri nets with a \emph{semantics} attached to them. A
semantics for a Petri net is a functor from its category of executions
to some other monoidal category $\Semantics$.

A huge conceptual difference is that in~\cite{Genovese2020} this functor was required to be strict monoidal.
This point of view backed up the interpretation that a semantics ``attaches extra information to tokens'',
to be used by the transitions somehow.
In here, we require this functor to be \emph{lax-monoidal}:\footnote{A \emph{lax monoidal} functor between two monoidal categories $(\mathcal C, \boxtimes J), (\mathcal D,\otimes, I)$ is a functor $F : \mathcal C \to \mathcal D$ endowed with maps $m : FA\otimes FB \to F(A\boxtimes B)$ and $u : I \to FJ$ satisfying suitable coherence conditions; see \cite[Def. 3.1]{aguiar}. If $m,u$ are isomorphisms in $\mathcal D$, $F$ is called \emph{strong monoidal}. If just $u$ is an isomorphism, $F$ is called \emph{normal monoidal}.} lax-monoidality amounts to saying
that we can attach \emph{non-local} information to tokens: tokens may ``know'' something
about the overall state of the net and the laxator represents the process of ``tokens joining knowledge''.

In terms of mana construction, we want to endow each token
with a local ``knowledge'' of how much mana each transition
has available. Laxating amounts to consider ensembles of
tokens together -- as entangled, if you wish --
where their knowledge is merged.
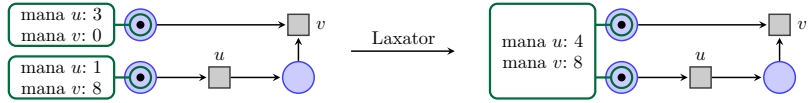
\begin{example}
	\begin{linenomath*}
		\begin{equation*}
			\scalebox{0.7}{
	\begin{tikzpicture}
		\definecolor{custom}{rgb}{0.0, 0.42, 0.24}
		\begin{scope}[xshift=0]
			\begin{pgfonlayer}{nodelayer}
				\node [place,tokens=1] (1a) at (-1.5,1) {};
				\node [place,tokens=1] (1b) at (-1.5,0) {};
				\node [place,tokens=0] (3a) at (1.5,0) {};

				\node[transition, label=above:$u$] (2a) at (0,0) {};
				\node[transition, label=right:$v$] (4a) at (1.5,1) {};

				\draw[very thick, custom] (-1.5,0) circle (0.2);
				\draw[very thick, custom] (-1.7,0) -- (-2,0);
				\draw[very thick, custom] (-1.5,1) circle (0.2);
				\draw[very thick, custom] (-1.7,1) -- (-2,1);

				\draw[very thick, rounded corners, custom] (-4,1.4) rectangle (-2,0.6);
				\draw[very thick, rounded corners, custom] (-4,0.4) rectangle (-2,-0.4);

				\node (lab1) at (-3,1.2) {mana $u$: 3};
				\node (lab2) at (-3,0.8) {mana $v$: 0};

				\node (lab3) at (-3,0.2) {mana $u$: 1};
				\node (lab4) at (-3,-0.2) {mana $v$: 8};

			\end{pgfonlayer}
			\begin{pgfonlayer}{edgelayer}
				\draw[style=inarrow, thick] (1a) to (4a);
				\draw[style=inarrow, thick] (1b) to (2a);
				\draw[style=inarrow, thick] (2a) to (3a);
				\draw[style=inarrow, thick] (3a) to (4a);
			\end{pgfonlayer}
		\end{scope}

		\begin{scope}[xshift=260]
			\begin{pgfonlayer}{nodelayer}
				\node [place,tokens=1] (1a) at (-1.5,1) {};
				\node [place,tokens=1] (1b) at (-1.5,0) {};
				\node [place,tokens=0] (3a) at (1.5,0) {};

				\node[transition, label=above:$u$] (2a) at (0,0) {};
				\node[transition, label=right:$v$] (4a) at (1.5,1) {};

				\draw[very thick, custom] (-1.5,0) circle (0.2);
				\draw[very thick, custom] (-1.7,0) -- (-2,0);
				\draw[very thick, custom] (-1.5,1) circle (0.2);
				\draw[very thick, custom] (-1.7,1) -- (-2,1);

				\draw[very thick, rounded corners, custom] (-4,1.4) rectangle (-2,-0.4);

				\node (lab1) at (-3,0.7) {mana $u$: 4};
				\node (lab2) at (-3,0.3) {mana $v$: 8};

			\end{pgfonlayer}
			\begin{pgfonlayer}{edgelayer}
				\draw[style=inarrow, thick] (1a) to (4a);
				\draw[style=inarrow, thick] (1b) to (2a);
				\draw[style=inarrow, thick] (2a) to (3a);
				\draw[style=inarrow, thick] (3a) to (4a);
			\end{pgfonlayer}
		\end{scope}

		\draw[inarrow, thick] (2.5,0.5) --node[above, midway] {Laxator} (4.5,0.5);
	\end{tikzpicture}
}
		\end{equation*}
	\end{linenomath*}
	If token $a$ knows
	that transition $u$ has $3$ mana left, and token $b$
	knows that transitions $u$ and $v$ have $1$ and $8$ mana
	left, respectively, then tokens $a$ and $b$, considered
	together, know that transitions $u$ and $v$ have $3+1 = 4$
	and $0+8 = 8$ mana left, respectively.
\end{example}
\begin{definition}[Non-local semantics -- common-token philosophy]\label{def: non-local semantics common-token philosphy}
	Let $N$ be a Petri net and let $\Semantics$ be a monoidal category.
	A \emph{Petri net with a non-local commutative semantics} is a couple $\NetSem{N}$, with
	$\Fun{N}$ a lax-monoidal functor $\Comm{N} \to \Semantics$.
	A morphism $\NetSem{N} \to \NetSem{M}$ of Petri nets
	with commutative semantics is a strict monoidal functor $\Comm{N} \xrightarrow{F} \Comm{M}$.

	We denote the category of Petri nets with non-local commutative semantics with $\PetriS{\Semantics}$.
\end{definition}
We now provide an external version of the mana construction.
\begin{notation}\label{spansi}
	We denote with $\Span$ the bicategory of sets, spans and span morphisms between them.\footnote{See \cite[Def. 1.1]{Benabou1967} for the definition of bicategory; intuitively, in a bicategory, one has objects (0-cells), 1-cells and 2-cells, and composition of 1-cells is associative and unital up to some specified invertible 2-cells $F(GH) \cong (FG)H$ and $F1\cong F\cong 1F$.}
	Recall that a morphism $A \to B$ in $\Span$ consists of a set $S$ and a pair of functions
	$A \leftarrow S \rightarrow B$. When we need to notationally extract this information from
	$f$, we write $A \xleftarrow{f_1} S_f \xrightarrow{f_2} B$.
	We sometimes consider a span as a morphism $f: S_f \to A \times B$, thus we may
	write $f(s)  = (a,b)$ for $s \in S_f$ with $f_1(s) = a$ and $f_2(s) = b$.
	Recall moreover that a 2-cell in $\Span$ $f \Rightarrow g$ is a function $\theta:S_f \to S_g$
	such that $f = g\circ \theta$.
\end{notation}
Observe that there is nothing in the previous definition of $\Span$ that requires the objects to be mere sets; in particular, we will later employ the following variation on Notation \ref{spansi}:
\begin{definition}[Spans of pointed sets]
	Define a bicategory $\Span_\bullet$ of \emph{spans of pointed sets} objects the pointed sets, $(A,a)$ where $a\in A$ is a distinguished element; composition of spans is as expected
\end{definition}
\begin{remark}
	This is in turn just a particular case of a more general construction: let $\mathcal C$ be a category with pullbacks; then, there is a bicategory $\Span\, \mathcal C$ having 1-cells the spans $A \leftarrow X \to B$ of morphisms of $\mathcal C$, and where a pullback of their adjacent legs defines the composition of 1-cells. Evidently, $\Span=\Span(\Set)$ and  $\Span_\bullet = \Span(\Set_\bullet)$, where $\Set_\bullet$ is the category of pointed sets $(A,a)$ and maps that preserve the distinguished elements of the domain and codomain. See \cite[§2]{dawson2010span} and \cite{dawson2004universal} for a way more general perspective on bicategories of the form $\Span \, \mathcal C$ and the universal property of the $\Span$ construction.
\end{remark}
\begin{definition}[External mana construction]\label{def: external mana construction}
	Given a Petri net $N: T \xrightarrow{s,t} \Msets{S}$, define
	the following functor $\Fun{N}: \Comm{N} \to \Span$:
	\begin{itemize}
		\item Each object $A$ of $\Comm{N}$ is mapped to the set $\Msets{T}$, the set of multisets over the transitions of $N$;
		\item Each morphism $A \xrightarrow{f} B$ is sent to the span $\Fun{N}f$ defined as:
		      \begin{linenomath*}
			      \begin{equation*}
				      \Msets{T} \xleftarrow{- \oplus \Mset{f}} \Msets{T} \Relbar \Msets{T}
			      \end{equation*}
		      \end{linenomath*}
		      With $\Mset{f}$ being the multiset counting how many
		      times each generating morphism of $\Comm{M}$ is used in $f$.
	\end{itemize}
\end{definition}
\begin{proposition}
	The functor of~\cref{def: external mana construction} is lax monoidal.
	Functors as in~\cref{def: external mana construction} form a subcategory of
	$\PetriSpan$, which we call $\PetriMana$.
\end{proposition}
\begin{proof}
	Functor laws are obvious:  $\Id{A}^{\Naturals}$
	is the empty multiset for each object $A$,
	hence $\Fun{N}\Id{A} = \Id{\Msets{T}}$.
	This correspondence preserves composition since
	\begin{linenomath*}
		\begin{equation*}
			\begin{tikzcd}
				T^\oplus \ar[r,equal]\pb\ar[d, "-\oplus g^{\mathbb N}"']& T^\oplus\ar[d, "-\oplus g^{\mathbb N}"] \ar[r,equal] & T^\oplus \\
				T^\oplus\ar[r,equal]\ar[d, "-\oplus f^{\mathbb N}"'] & T^\oplus & \\
				T^\oplus & &
			\end{tikzcd}
			=
			\begin{tikzcd}
				T^\oplus\ar[rr, equal]\ar[dd, "{-\oplus f^{\mathbb N}\oplus g^{\mathbb N}}"'] &  & T^\oplus \\
				&  & \\
				T^\oplus & &
			\end{tikzcd}
		\end{equation*}
	\end{linenomath*}
	%
	The laxator is the morphism $\Msets{S} \times \Msets{S} \xrightarrow{\oplus} \Msets{S}$
	that evaluates two multisets to their sum, embedded in a span.
	The naturality condition for the laxator reads:
	\begin{linenomath*}
		\begin{equation*}
			\begin{tikzcd}
				{\Msets{T} \times \Msets{T}} && {\Msets{T} \times \Msets{T}} \\
				{\Msets{T}} && {\Msets{T}}
				\arrow["{\Fun{N}f \times \Fun{N}g}", from=1-1, to=1-3]
				\arrow["{\Fun{N}{(f \oplus g)}}"', from=2-1, to=2-3]
				\arrow["{\oplus}"', from=1-1, to=2-1]
				\arrow["{\oplus}", from=1-3, to=2-3]
			\end{tikzcd}
		\end{equation*}
	\end{linenomath*}
	And the two morphisms from $\Msets{T} \times \Msets{T} \to \Msets{T}$ are:
	%
	%
	\begin{linenomath*}
		\begin{equation*}
			\adjustbox{scale=0.9}{
				\begin{tikzcd}
					T^\oplus\times T^\oplus \ar[r, equal]\ar[d, equal]& T^\oplus \times T^\oplus\ar[d,equal]\ar[r,"\oplus"]  & T^\oplus \\
					T^\oplus \times T^\oplus \ar[r,equal]\ar[d, "(-\oplus f^{\mathbb N})\times(-\oplus g^{\mathbb N})"]& T^\oplus \times T^\oplus & \\
					T^\oplus \times T^\oplus
				\end{tikzcd}}
			=
			\adjustbox{scale=0.9}{
				\begin{tikzcd}
					T^\oplus \times T^\oplus \ar[r, "\oplus"]\ar[d, "(-\oplus f^{\mathbb N})\times(-\oplus g^{\mathbb N})"']& T^\oplus \ar[d, "-\oplus f^{\mathbb N}\oplus g^{\mathbb N}"]\ar[r, equal] & T^\oplus \\
					T^\oplus\times T^\oplus\ar[r, "\oplus"']\ar[d,equal] & T^\oplus & \\
					T^\oplus\times T^\oplus
				\end{tikzcd}}
		\end{equation*}
	\end{linenomath*}
	which evidently coincide. Interaction with the associators, unitors
	and symmetries of the monoidal structure is guaranteed by the fact that they are all identities in $\Comm{N}$.
\end{proof}
The external mana construction has the advantage of keeping the
reaction layer and the mana layer separated completely. In this setting, we
say that a marking of the net is a couple $(X,u)$, with $X$ an object of
$\Comm{N}$ and $u \in \Msets{T}$ representing the initial distribution
of mana for our transitions. A transition $X \xrightarrow{f} Y$ is again a generating
morphism of $\Comm{N}$, and we say that it is enabled if
$\Fun{N}f_1$ hits $u$, or, more explicitly, if $u \ominus \Mset{f}$ is defined.
Since $\Mset{f}$ for $f$ a morphism generator is defined to be $0$ everywhere and $1$
on $f$, this amounts to say that $f$ is enabled when $u(f) - 1 \geq 0$. In that case,
the resulting marking after the firing is $(Y, u(f)-1)$:
Each firing just decreases the mana of the firing
transition by $1$.
\begin{example}\label{ex: external mana construction}
	Consider the net
	\begin{linenomath*}
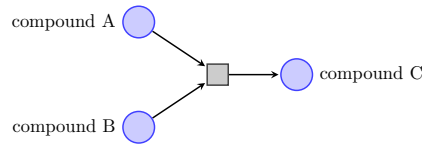
		\begin{equation*}
			\scalebox{0.7}{
	\begin{tikzpicture}
		\begin{scope}[xshift=0]
			\begin{pgfonlayer}{nodelayer}
				\node [place,tokens=0, label=left:compound A] (1a) at (-1.5,1) {};
				\node [place,tokens=0, label=left:compound B] (1b) at (-1.5,-1) {};
				\node [place,tokens=0, label=right:compound C] (3a) at (1.5,0) {};

				\node[transition] (2b) at (0,0) {};
			\end{pgfonlayer}
			\begin{pgfonlayer}{edgelayer}
				\draw[style=inarrow, thick] (1a) to (2b);
				\draw[style=inarrow, thick] (1b) to (2b);
				\draw[style=inarrow, thick] (2b) to (3a);
			\end{pgfonlayer}
		\end{scope}

	\end{tikzpicture}
}
		\end{equation*}
	\end{linenomath*}
	In the marking $(A \oplus B, 2)$, the transition is enabled.
	The resulting marking will be $(C,1)$. The transition is \emph{not}
	enabled in the marking $(A \oplus B, 0)$ or $(A, 4)$.
\end{example}
\subsection{Internalization}
Having given two different definitions of endowing a net with mana,
it seems fitting to say how the two are connected. As we already stressed,
we abide by the praxis already established in~\cite{Genovese2020} and prove that the
external and internal mana constructions describe the same thing from
different points of view :
\begin{theorem}\label{thm: internalization}
	Let $\NetSem{N}$ be an object of $\PetriMana$.
	The category $\CommM{N}$ of~\cref{def: internal mana construction}
	is isomorphic to the category of elements $\GrothendieckS{N}$.\footnote{The category of elements of a functor $F : \mathcal C \to \mathbf{Set}$ is defined having objects the pairs $(C,x)$, where $x\in FC$, and morphisms $(C,x)\to (C',x')$ the morphisms $u : C \to C'$ such that $Fu$ sends $x$ into $x'$. See \cite[12.2]{barr1990category}, where this is called the \emph{Grothendieck construction} performed on $F$. Here we need to tweak this construction in order for it to make sense for \emph{lax} functors valued in $\Span$, using essentially the same technique in \cite{Pavlovic1997}.} Explicitly:
	\begin{itemize}
		\item Objects of $\GrothendieckS{N}$ are couples $(X,x)$ where
		      $X$ is a object of $\CommM{N}$ and $x \in \Fun{N}X$.

		\item Morphisms $(X,x) \to (Y,y)$ of $\GrothendieckS{N}$ are
		      morphisms $(f,s)$ with $f: X \to Y$ of $\CommM{N}$ and $s$ such that $\Fun{N}f s = (x,y)$.
	\end{itemize}
\end{theorem}
\begin{proof}
	First of all, we need to define a commutative strict monoidal structure on $\GrothendieckS{N}$.
	Given the particular shape of $\Fun{N}$, the objects of its category of elements are pairs where the first component is a multiset on the places of $N$ and the second one is a multiset on its transition.
	Hence we can define:
	\begin{linenomath*}
		\begin{equation*}
			(C, x) \boxtimes (D,y) := (C \oplus D, x \oplus y)
		\end{equation*}
	\end{linenomath*}
	(Note that in order to obtain an element in $\Fun{N}(C \oplus D)$, we have implicitly
	applied the laxator $\oplus: \Fun{N}C \times \Fun{N}D \to \Fun{N}(C \oplus D)$ to the
	elements in the second coordinate.) Commutativity of $\boxtimes$ follows from the commutativity of $\oplus$.
	On morphisms, if we have $(A_1,x_1) \xrightarrow{(f_1, s_1)} (B_1,y_1)$ and
	$(A_2,x_2) \xrightarrow{(f_2, s_2)} (B_2,y_2)$ then it is $\Fun{N}f_1 s_1 = (x_1,y_1)$ and
	$\Fun{N}f_2 s_2 = (x_2,y_2)$, and hence by naturality of the laxator
	$\Fun{N}(f_1 \oplus f_2)(s_1,s_2) = ((x_1 \oplus x_2 ), (y_1 \oplus y_2))$,
	allowing us to set $f_1 \boxtimes f_2 = f_1 \oplus f_2$. Associators and unitors
	are defined as in $\Comm{N}$.

	Now we prove freeness: by definition, objects are
	a free monoid generated by couples $(p,\TensorUnit)$ and $(\TensorUnit,u)$
	with $p$ a generating object of $\Comm{N}$ (a place of $N$), $u$ a generating
	morphism of $\CommM{N}$ (a transition of $N$), and $\TensorUnit$ the tensor unit.
	These generators are in bijection with the coproduct of places and transitions of $N$.
	As such, the monoid of objects of $\GrothendieckS{N}$ is isomorphic to the one of $\CommB{N}$.

	On morphisms, notice that every morphism in $\GrothendieckS{N}$ can be written
	univocally -- modulo the axioms of a commutative strict monoidal category -- as
	a composition of monoidal products of identities and morphisms of the form
	$(A,u) \xrightarrow{(f,u)} (B, u')$, with $f$ a morphism generator in $\Comm{N}$ and
	$u = u' \oplus \Mset{f}$.

	The isomorphism between $\GrothendieckS{N}$ and $\CommB{N}$ follows by observing that the following mappings
	between objects and morphism generators are bijections:
	\begin{linenomath*}
		\begin{align*}
			\left(A, u \right) & \mapsto A \oplus u                                    \\
			(A,u) \xrightarrow{(f,u)} (B,u')
			                   & \mapsto A \oplus u \xrightarrow{f}B \oplus u'\qedhere
		\end{align*}
	\end{linenomath*}
\end{proof}
\begin{example}
	The internalization of the net in~\cref{ex: external mana construction}
	gives exactly the net of~\cref{ex: internal mana construction}.
\end{example}

\section{Extending the mana construction}\label{ref: extending the mana construction}
Focusing more on the external mana construction
of~\cref{def: external mana construction},
we realize that it is somehow restrictive: it makes sense to map each generating object
of an FCSMC $\CommM{N}$ to the set of
multisets over the transitions of $N$. This construction captures
the idea of endowing each transition with an extra
place representing its mana. On the other hand,
the only requirement we would expect on morphisms
is that, to fire, a transition must consume mana only from its mana pool. In~\cref{def: external mana construction}
we do much more than this, hardcoding that
``one firing = one mana'' in the structure of the functor.

The act of replacing the
mapping on morphisms in~\cref{def: external mana construction}
with the following span provides a reasonable generalization of the previous perspective:
\begin{linenomath*}
	\begin{equation*}
		\Msets{T} \xleftarrow{- \oplus (\alpha \odot \Mset{f})} \Msets{T} \xrightarrow{- \oplus (\beta_f)} \Msets{T}
	\end{equation*}
\end{linenomath*}
with $\alpha$ and $\beta_f$ arbitrary multisets.
In doing so, the only thing we are disallowing in our new definition
is for transitions to consume mana of other transitions: each
transition may use only the mana in its pool. Still, it
is now possible for transitions to:
\begin{itemize}
	\item Fire without consuming mana;
	\item Consume more than $1$ unit of mana to fire;
	\item Produce mana -- also for other transitions -- upon firing.
\end{itemize}
These are all good conditions in practical applications. The first
models chemical reactions that do not need any additional compound to work; the second aims to model reactions that need more than one molecule of a given compound to work; the third models both catalysts -- which
completely regenerate their mana at the end of the reactions they aid -- and
reactions that produce, as byproducts, enzymes needed by other reactions.
\begin{example}\label{ex: generalized mana net}
	It is worth giving an explicit description of how the internalized version of a net, as in our attempted generalized definition, looks like. In the picture below, each transition has its mana, but now this
	mana does not have to be necessarily used, as for transition $u_1$,
	or can be used more than once, as for transition $u_2$.
	Furthermore, transitions such as $u_3$ regenerate their mana
	after firing (catalysts), while transitions such as $u_2$ and $u_4$
	produce mana for each other in a closed loop. $u_4$ also produces more than one kind of mana as a byproduct of its firing. It is worth noticing that
	this formalism allows to model nets that never run out of mana, and that
	we think of as ``self-sustaining''~\cite{Latelier2006}.
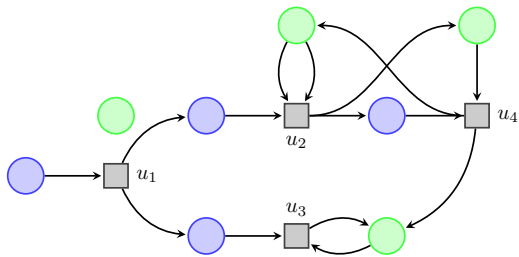
	\begin{linenomath*}
		\begin{equation*}
			\scalebox{0.8}{
	\begin{tikzpicture}
		\begin{scope}[xshift=0]
			\begin{pgfonlayer}{nodelayer}
				\node [place,tokens=0] (p1) at (0,0) {};
				\node [place,tokens=0] (p2) at (3,1) {};
				\node [place,tokens=0] (p3) at (3,-1) {};
				\node [place,tokens=0] (p4) at (6,1) {};

				\node [manaplace, tokens=0] (m1) at (1.5,1) {};
				\node [manaplace, tokens=0] (m2) at (4.5,2.5) {};
				\node [manaplace, tokens=0] (m3) at (6,-1) {};
				\node [manaplace, tokens=0] (m4) at (7.5,2.5) {};

				\node[transition, label=right:$u_1$] (t1) at (1.5,0) {};
				\node[transition, label=below:$u_2$] (t2) at (4.5,1) {};
				\node[transition, label=above:$u_3$] (t3) at (4.5,-1) {};
				\node[transition, label=right:$u_4$] (t4) at (7.5,1) {};
			\end{pgfonlayer}
			\begin{pgfonlayer}{edgelayer}
				\draw[style=inarrow, thick] (p1) to (t1);
				\draw[style=inarrow, thick, bend left] (t1) to (p2);
				\draw[style=inarrow, thick, bend right] (t1) to (p3);
				\draw[style=inarrow, thick] (p2) to (t2);
				\draw[style=inarrow, thick] (p3) to (t3);
				\draw[style=inarrow, thick] (t2) to (p4);
				\draw[style=inarrow, thick] (p4) to (t4);

				\draw[style=inarrow, thick, bend left] (t3) to (m3);
				\draw[style=inarrow, thick, bend left] (m3) to (t3);

				\draw[style=inarrow, thick, bend left] (m2) to (t2);
				\draw[style=inarrow, thick, bend right] (m2) to (t2);
				\draw[style=inarrow, thick, out=0, in=180] (t2) to (m4);

				\draw[style=inarrow, thick] (m4) to (t4);
				\draw[style=inarrow, thick, out=180, in=0] (t4) to (m2);
				\draw[style=inarrow, thick, bend left] (t4) to (m3);

			\end{pgfonlayer}
		\end{scope}

	\end{tikzpicture}
}
		\end{equation*}
	\end{linenomath*}
\end{example}
When looking at technicalities, unfortunately, things are not so easy.
Defining $\alpha$ and the family $\beta_f$
so that functorial laws are respected is tricky. Luckily enough, we do not need to do
so explicitly. Indeed, we can generalize the internal mana-net construction
of~\cref{def: internal mana construction} to the following one, that subsumes nets
as in~\cref{ex: generalized mana net}:
\begin{definition}[Generalized internal mana construction]\label{def: generalized internal mana construction}
	Let $N: T \xrightarrow{s,t} \Msets{S}$ be a Petri net,
	and consider $\Comm{N}$, its corresponding FCSMC.
	A \emph{generalized internal mana construction for $N$}
	is any FCSMC $\CommM{N}$ such that:
	\begin{itemize}
		\item The generating objects of $\CommM{N}$ are the coproduct of the generating objects of $\Comm{N}$ and $T$;
		\item Generating morphisms
		      \begin{linenomath*}
			      \begin{equation*}
				      A_1 \oplus \dots \oplus A_n \xrightarrow{u} B_1 \oplus \dots \oplus B_m
			      \end{equation*}
		      \end{linenomath*}
		      in $\Comm{N}$ are in bijection with generating morphisms in $\CommM{N}$:
		      \begin{linenomath*}
			      \begin{equation*}
				      A_1 \oplus \dots \oplus A_n \oplus U_1 \xrightarrow{u} U_2 \oplus B_1 \oplus \dots \oplus B_m
			      \end{equation*}
		      \end{linenomath*}
		      With $U_1$ a multiset over $T$ being $0$ on any $u' \neq u$, and $U_2$ being an arbitrary multiset over $T$.
	\end{itemize}
\end{definition}
Notice moreover that, for each generalized mana-net $\CommM{N}$,
we obtain a strict monoidal functor $F: \CommM{N} \to \Comm{N}$ as
in~\cref{prop: comonad}: we send generating objects of $\Comm{N}$ to themselves,
all the other generating objects to the monoidal unit and generating morphisms to themselves.
We keep calling $F$ the \emph{counit of $\CommM{N}$}, even if it won't be in
general true that we still get a comonad.

Counits can be turned into functors $\Comm{N} \to \Span$ using
a piece of categorical artillery called the \emph{Grothendieck construction} (or the \emph{category of elements} construction).
\begin{theorem}[Grothendieck construction, \cite{Pavlovic1997}]\label{thm: total category}
	Let $\CategoryC$ be a category. Then, there is an equivalence $\Cat/\CategoryC \simeq \Cat_l [\CategoryC,\Span]$,
	with $\Cat_l [\CategoryC,\Span]$ being the category of lax functors $\CategoryC \to \Span$.
	A functor $F: \CategoryD \to \CategoryC$ defines a functor $\Gamma F: \CategoryC \to \Span$ as follows:
	\begin{itemize}
		\item On objects, $C$ is mapped to the set $\Suchthat{D \in \CategoryD}{FD = C}$;
		\item On morphisms, $C \xrightarrow{f} C'$ is mapped to the span
		      \begin{linenomath*}
			      \begin{equation*}
				      \Suchthat{D \in \CategoryC}{FD = C} \xleftarrow{s} \Suchthat{g \in \CategoryD}{Ff = g} \xrightarrow{t} \Suchthat{D \in \CategoryC}{FD = C'}
			      \end{equation*}
		      \end{linenomath*}
	\end{itemize}
	The other way around, regarding  $\CategoryC$ as a locally discrete bicategory and letting $F: \CategoryC \to \Span$ be a lax functor,
	$F$ maps to the functor $\Sigma_F$, from the pullback (in $\Cat$) below:
	\begin{linenomath*}
		\begin{equation*}
			\begin{tikzcd}
	\Grothendieck{F} \ar[r]\ar[d, "\Sigma_F"'] \pb & \Span_\bullet\ar[d, "U"] \\
	\CategoryC \ar[r, "F"'] & \Span
\end{tikzcd}

		\end{equation*}
	\end{linenomath*}
	where $\Span_\bullet$ is the bicategory of spans between pointed sets,
	and $U$ is the forgetful functor.

	More concretely, $\Grothendieck{F}$ is defined as the category (all 2-cells are identities, due to the 2-discreteness of $\CategoryC$) having
	\begin{itemize}
		\item 0-cells of $\Grothendieck{F}$ are couples $(X,x)$ where
		      $X$ is a 0-cell of $\CategoryC$ and $x \in FX$;
		\item 1-cells $(X,x) \to (Y,y)$ of $\Grothendieck{F}$ are couples $(f,s)$
		      where $f: X \to Y$ is a 1-cell of $\CategoryC$ and $s \in S_{Ff}$ with
		      $Ff(s) = (x,y)$. Representing a span as a function $(S,s) \to (X\times Y, (x,y))$ between (pointed) sets, a morphism is a pair $(f,s)$ such that $Ff : s\mapsto (x,y)$.
	\end{itemize}
	Finally, the categories $\Grothendieck{F}$ and $\CategoryD$ are isomorphic.
\end{theorem}
This result is a particular case of a more general correspondence between
slice categories and lax normal functors to the category of profunctors~\cite{Loregian2020}, which is
well-known in category theory and dates back to Bénabou~\cite{Benabou1967, Benabou2000}.
It gives an entirely abstract way to
switch from/to and define internal/external semantics for mana-nets.
Indeed, with a proof
partly similar to the one carried out in our Theorem \ref{thm: internalization}, we can show that:
\begin{proposition}
	Monoidality of $\CommM{N} \xrightarrow{F} \Comm{N}$ implies $\Gamma F$ is lax-monoidal.
\end{proposition}
We can thus define the external semantics of a generalized mana-net by
applying $\Gamma$ to $F$.
A generalization of~\cref{thm: internalization}
then holds by definition.

Summing up, we showed that the mana-net construction can be generalized to more practical applications and the correspondence between a ``na\"ive'' internal semantics and a ``type-aware'' external one is still preserved. The evident price we have to pay for our generalization is that our external semantics is now not just lax-monoidal but lax-monoidal-lax.

\section{Conclusion and future work}\label{sec: conclusion and future work}
In this work, we introduced a new notion of Petri net
where transitions come endowed with ``mana'', a quality
representing how many times a transition will be able to
fire before losing its effectiveness. We believe this may
be especially useful in modelling chemical processes mediated
by enzymes that degrade over time.

Importantly, we showed how a categorical point of view
on the matter allows to give two different definitions:
A na\"ive, ``hands-on'' one, that we called \emph{internal},
and a type-aware, functorial one, that we called \emph{external},
which we proved to be two sides of the same coin.

Indeed, the equivalence between internal and external semantics is
the consequence of a much more profound result in category theory,
connecting slice categories and categories of lax monoidal functors.
We were able to rely on this result to generalize our mana-nets further,
while keeping the equivalence between the internal and external
points of view.

We believe that further generalizations of the external
semantics presented here may prove valuable to produce
categorical semantics for nets with
\emph{inhibitor arcs}~\cite{Agerwala1974}. An inhibitor arc is an
input arc to a transition that is enabled only when there are
no tokens in their place. This concept is powerful enough to
turn Petri nets into a Turing-complete model of computation~\cite{Zaitsev2012, Zaitsev2014}.

Indeed, we notice that by relaxing~\cref{def: external mana construction}
to allow \emph{any} span $\Msets{T} \to \Msets{T}$,
we can model situations where a transition can fire just if it has no mana
(e.g., we can map transition $f$ to a span that is only defined
when its source multiset has value $0$ on $f$). The similarities in behaviour with inhibitor arcs are evident and constitute a direction of future
work that we will surely pursue. The various technicalities involved
are nevertheless tricky and necessitate a careful investigation.

Another direction of future work is about implementing
the ideas at this moment presented using already available
category theory libraries, such as~\cite{Genovese2019d}.

\printbibliography
\end{document}